\documentclass[11pt]{article}

\usepackage{amssymb} 
\usepackage[intlimits]{amsmath}
\usepackage{amsfonts}
\usepackage{amsthm}

\usepackage{graphicx}
\usepackage{pst-pdf}
\usepackage{pdfpages}

\usepackage{mathptmx}

\usepackage{hyperref}

\usepackage[a4paper]{geometry}

\let\oldsqrt\sqrt
\def\sqrt{\mathpalette\DHLhksqrt}
\def\DHLhksqrt#1#2{%
\setbox0=\hbox{$#1\oldsqrt{#2\,}$}\dimen0=\ht0
\advance\dimen0-0.2\ht0
\setbox2=\hbox{\vrule height\ht0 depth -\dimen0}%
{\box0\lower0.4pt\box2}}

\allowdisplaybreaks

\newcommand{\R}{\mathbb{R}} 

\newcommand{\dist}{\textnormal{dist}} 
\newcommand{\diam}{\textnormal{diam}} 
\newcommand{\supp}{\textnormal{supp}} 
\newcommand{\essinf}{\textnormal{essinf}} 

\renewcommand{\phi}{\varphi}

\newcommand{\cW}{{\mathcal W}}

\newcommand{\rn}{\R^N}

\theoremstyle{definition}
\newtheorem{defi}{Definition}[section]
\newtheorem{remark}[defi]{Remark}

\theoremstyle{plain} 
\newtheorem{thm}[defi]{Theorem}

\newtheorem{lemma}[defi]{Lemma}

\newtheorem{cor}[defi]{Corollary}

\theoremstyle{definition}

\numberwithin{equation}{section}

 \title{Strong comparison principle for the fractional $p$-Laplacian and applications to starshaped rings}
\author{
	\ Sven Jarohs\footnote{Goethe-Universit\"at Frankfurt, Germany, jarohs@math.uni-frankfurt.de.}
}

\date{\today}

\begin{document}
\maketitle
\begin{abstract}
	In the following we show the strong comparison principle for the fractional $p$-Laplacian, i.e. we analyze 
	\[
	(P)\quad\left\{\begin{aligned}
		(-\Delta)^s_pv+q(x)|v|^{p-2}v&\geq 0&&\text{ in $D$}\\
		(-\Delta)^s_pw+q(x)|w|^{p-2}w&\leq 0&&\text{ in $D$}\\
		v&\geq w&&\text{in $\R^N$}
	\end{aligned}\right.
	\]
	where $s\in(0,1)$, $p>1$, $D\subset \R^N$ is an open set, and $q\in L^{\infty}(\R^N)$ is a nonnegative function. Under suitable conditions on $s,p$ and some regularity assumptions on $v,w$ we show that either $v\equiv w$ in $\R^N$ or $v>w$ in $D$. Moreover, we apply this result to analyze the geometry of nonnegative solutions in starshaped rings and in the half space.
\end{abstract}

{\footnotesize
\begin{center}
\end{center}
}

\section{Introduction}

In the following we investigate an ordered pair of functions $v,w:\R^N\to \R$ which are sub- and supersolution of the equation
\begin{equation}\label{main-prob2}
(-\Delta)^s_p u +q(x)|u|^{p-2}u=g \qquad \text{ in $D$,}
\end{equation}
where $s\in(0,1)$, $p>1$, $q\in L^{\infty}(D)$ is a nonnegative function, $g\in L^{p'}(D)$ with $p'=\frac{p}{p-1}$ the conjugate of $p$, and $(-\Delta)^s_p$ is the $s$-fractional $p$-Laplacian (up to a constant). Recall that for suitable $(s,p)$ and some smoothness conditions on $u$ we may write (see \cite[Proposition 2.12]{IMS14})
\[
 (-\Delta)^s_p u(x)=\lim_{\epsilon\to 0^+}\int_{\R^N\setminus B_{\epsilon}(x)}\frac{|u(x)-u(y)|^{p-2}(u(x)-u(y))}{|x-y|^{N+sp}}\ dy,\;\; x\in \R^N.
\]
In order to derive a strong comparison principle for the fractional $p$-Laplacian we use a weak setting. We denote by $W^{s,p}(\R^N)$ as usual the fractional Sobolev space of order $(s,p)$ given by
\begin{equation}\label{wsp-rn}
W^{s,p}(\R^N)=\Bigg\{u\in L^{p}(\R^N)\;:\; \int_{\R^N}\int_{\R^N}\frac{|u(x)-u(y)|^p}{|x-y|^{N+sp}}\ dxdy<\infty\Bigg\}
\end{equation}
and for an open set $D\subset \R^N$ we denote 
\begin{equation}\label{wsp-d}
\cW^{s,p}_0(D):=\{u\in W^{s,p}(\R^N)\;:\; u\equiv 0 \text{ on $\R^{N}\setminus D$}\}.
\end{equation}
For an introduction into fractional Sobolev spaces we refer to \cite{NPV11}. Finally, we also use the space 
\begin{equation}\label{wsp-gen}
\tilde{W}^{s,p}(D):=\{u\in L^{p}_{loc}(\R^N)\;:\; \int_{D}\int_{\R^N} \frac{|u(x)-u(y)|^{p}}{|x-y|^{N+sp}}\ dxdy<\infty\}
\end{equation}
to admit function with a certain growth at infinity. Given an open set $D\subset \R^N$, $q\in L^{\infty}(D)$, a function $v\in \tilde{W}^{s,p}(D)$ is called a supersolution of
\eqref{main-prob2}
if for all nonnegative $\varphi\in \cW_0^{s,p}(D)$ with compact support in $\R^N$ we have
\begin{equation}\label{weak-formulation}
 \int_{\R^{N}}\int_{\R^{N}}\frac{|v(x)-v(y)|^{p-2}(v(x)-v(y))(\varphi(x)-\varphi(y))}{|x-y|^{N+sp}}\ dxdy + \int_{D} q|v|^{p-2}v\varphi\ dx \geq \int_Dg\phi\ dx.
\end{equation}
Similarly we call $v$ a subsolution of \eqref{main-prob2}, if $-v$ is a supersolution of \eqref{main-prob2}. If $v$ is a sub- and a supersolution of \eqref{main-prob2} \underline{and} $v\in \cW^{s,p}_0(D)$, then we call $v$ a solution of \eqref{main-prob2}. We note that indeed the left-hand side in \eqref{weak-formulation} is well-defined as is shown in Lemma \ref{well-defined} below.\\

Equations involving the fractional Laplacian, that is the case of $p=2$, have been studied extensively in recent years (see e.g. \cite{BV16} and the references in there), whilst for its nonlinear counterpart there are still several unanswered questions. Existence of solutions and their regularity has been treated in \cite{DKP16,KKP16a,IMS13,IMS14}. In particular, the question of existence of nontrivial solutions to problem \eqref{main-prob2} in the case $q=0$ with nontrivial outside data has been studied in \cite{DKP16,KKP16a}. Let us also mention \cite{LL14}, where the Rayleigh quotient associated to $(-\Delta)^s_p$ has been studied and \cite{KKP16b}, which analyzes the obstacle problem associated with the fractional $p$-Laplacian. In this work we prove a strong comparison principles for equations of type \eqref{main-prob2} and apply this to equations in starshaped rings and in the half space.

\begin{thm}[Strong comparison principle]\label{smp}
	Let $s\in(0,1)$, $p>1$, $D\subset \R^N$ be an open set, $q\in L^{\infty}(D)$, $q\geq0$, $g\in L^{p'}(D)$, where $p'=\frac{p}{p-1}$, and let $v,w\in \tilde{W}^{s,p}(D)$ be such that $v$ is a supersolution and $w$ a subsolution of \eqref{main-prob2} with $v\geq w$. If one of the following holds
   \begin{enumerate}
   	\item $\frac{1}{1-s}<p\leq2$ and $v\in L^{\infty}(\R^N)$ or $w\in L^{\infty}(\R^N)$, or
   	\item $p\geq 2$ and for some $\alpha\in(0,1]$ with $\alpha(p-2)>sp-1$ we have $v\in C_{loc}^{\alpha}(D)\cap L^{\infty}(\R^N)$ or $w\in C^{\alpha}_{loc}(D)\cap L^{\infty}(\R^N)$,
   \end{enumerate}
	then either $v=w$ a.e. in $\R^N$ or 
	\[
	\underset{K}\essinf\ (v-w)>0 \quad\text{ for all $K\subset\subset D$.}
	\]
\end{thm}

The weak comparison principle for the fractional $p$-Laplacian with $q=0$ goes back to \cite{LL14} (see also \cite{IMS14,KKP16a}). However, the validity of a strong comparison principle is already a delicate question in the case $s=1$, i.e. the case of the classical $p$-Laplacian. We refer here to the works \cite{PS04} and \cite{RS07}. Note that in the above nonlocal case, neither $v$ or $w$ need to be solutions and indeed to achieve such a statement we strongly use the nonlocal structure of the fractional $p$-Laplacian. In the case $p=2$, of course, the strong comparison principle follows from the strong maximum principle by linearity (see e.g. \cite{FJ13}). But in general, when $p\neq 2$, the strong maximum principle for the fractional $p$-Laplacian does not imply the strong comparison principle due to the nonlinear structure of the operator. For the strong maximum principle and a Hopf type lemma for the fractional $p$-Laplacian we refer to the recent work \cite{PQ17}.

For an application of Theorem \ref{smp} we investigate bounded nonnegative solutions of \eqref{main-prob2} in starshaped rings, i.e. we analyze
\begin{equation}\label{main-prob}
 \left\{\begin{aligned}
  (-\Delta)^s_pu+q(x)|u|^{p-2}u&=0&&\text{ in $D=D_0\setminus \overline{D}_1$;}\\
u&=0&&\text{ on $\R^N\setminus D_0$;}\\
u&=1&& \text{ on $D_1$,}
 \end{aligned}
\right.
\end{equation}
where $D_0,D_1\subset \R^N$ are open sets with $0\in \overline{D_1}\subset D_0$. For our main statement, we recall that a subset $A$ of $\R^N$ is said {\em starshaped with respect to the point} $\bar x\in A$ if for every $x\in A$ the segment $(1-s)\bar x+s x$, $s\in[0,1]$, is contained in $A$. If $\bar x=0$ (as we can always assume up to a translation), we simply say that $A$ is {\em starshaped}, meaning that for every $x\in A$ we have $sx\in A$ for $s\in[0,1]$, or equivalently
\begin{equation}\label{stardef}
A\text{ is starshaped if }\,\,sA\subseteq A\,\,\text{ for every }s\in[0,1]\,.
\end{equation}
$A$ is said {\em strictly starshaped} if $0$ is in the interior of $A$ and any ray starting from $0$ intersects the boundary of $A$ in only one point.

By $U(\ell)$, $\ell \in \R$ we denote the superlevel sets of a function $u:\R^N\to \R$:
$$
U(\ell):=\{u\geq \ell\}=\{x\in\rn\,:\, u(x)\geq \ell\}\,.
$$

\begin{thm}\label{res1}
	Let $s\in(0,1)$, $p>1$ and $D=D_0\setminus \overline{D}_1$ with $D_0, D_1\subset \R^N$ open bounded sets such that $0\in D_1$ and $\overline{D}_1\subset D_0$. Let $q:D\to [0,\infty)$ such that
	\begin{enumerate}
		\item[(A1)] $q$ is a bounded Borel-function and
		\item[(A2)] for all $t\geq 1$ and $x\in \R^N$ such that $tx\in D$ we have $t^{sp}q(tx)\geq q(x)$.
	\end{enumerate}
	Moreover, let $u$ be a continuous bounded weak solution of \eqref{main-prob} such that $0\leq u\leq 1$ and assume $D_0$ and $D_1$ are starshaped, then the superlevel sets  $U(\ell)$ of $u$ are starshaped for $\ell\in (0,1)$.\\
	If in addition $D_0$ and $D_1$ are strictly starshaped sets and
	\begin{enumerate}
		\item $\frac{1}{1-s}<p\leq 2$ \underline{or}
		\item $p\geq 2$ and $u\in C^{\alpha}_{loc}(D)$ for some $\alpha\in(0,1]$ with $\alpha(p-2)>sp-1$
	\end{enumerate}
 then the superlevel sets $U(\ell)$ of $u$ are strictly starshaped for $\ell\in (0,1)$.
\end{thm}

\begin{remark}
\begin{enumerate}
\item The starshapedness of superlevel sets is indeed a consequence of the weak comparison principle, hence the assumptions are rather general in this case. To prove the strict starshapedness of superlevel sets, however, we need the strong comparison principle and hence stronger assumptions on $u$, $s$ and $p$ in view of Theorem \ref{smp}. We note that in the case $q\equiv 0$ existence and local H\"older regularity of solutions of \eqref{main-prob} has been discussed in \cite{DKP16} so Theorem \ref{res1} can be applied for $p\in(0,1)$, $s<\frac{p-1}{p}$ and for any $p\geq 2$, $s<\frac{1}{p}$. 
\item In the case $p=2$ neither the bounds on $s$ nor the regularity assumption on $u$ are necessary (see \cite{JKS17}).
\end{enumerate}
\end{remark}

Let us close this introduction with the following further result in half spaces (see also \cite{FW16,CLL17a,CLL17b} for similar results).

\begin{thm}\label{res3}
Let $s\in(0,1)$, $p>1$ and denote $\R^N_+:=\{x\in \R^N\;:\; x_1>0\}$. Moreover, let $q\in L^{\infty}(\R^N_+)$, $q\geq 0$ and let $u\in \cW^{s,p}_0(\R^N_+)\cap L^{\infty}(\R^N_+)$ be a nonnegative continuous function which satisfies
\begin{equation}\label{main-prob3}
(-\Delta)^s_pu+q(x)|u|^{p-2}u=0\quad \text{ in $\R^N_+$;}\qquad \lim_{|x|\to \infty} u(x)=0.
\end{equation}
If $q$ is increasing in the direction of $x_1$, i.e. $q(x+te_1)\geq q(x)$ for all $x\in \Omega$, $t\geq0$, and
\begin{enumerate}
	\item $\frac{1}{1-s}<p\leq 2$ \underline{or}
	\item $p\geq 2$ and $u\in C^{\alpha}_{loc}(\R^N_+)$ for some $\alpha\in(0,1]$ with $\alpha(p-2)>sp-1$,
\end{enumerate}
then $u\equiv 0$ on $\R^N$.
\end{thm}

The article is organized as follows. In Section \ref{prem} we give some basic properties on the involved function spaces and useful elementary inequalities. In Section \ref{comparison} we give the proof of a variant of a weak comparison principle and then prove Theorem \ref{smp}. The proofs of Theorem \ref{res1} and \ref{res3} are given in Section \ref{applications}.

\subsection*{Acknowledgement}

The author thanks Tadeusz Kulczycki, Paolo Salani, and Giampiero Palatucci for careful reading and discussions. Moreover, he thanks Tobias Weth for raising the question on the strong comparison principle.

\section{Preliminaries and notation}\label{prem}
We will use the following notation. For subsets $D,U \subset \R^N$ we denote by $D^c:=\R^N\setminus D$ the complement of $D$ in $\R^N$ and we write $\dist(D,U):= \inf\{|x-y|\::\: x \in D,\, y \in U\}$.  If $D= \{x\}$ is a singleton, we write $\dist(x,U)$ in place of $\dist(\{x\},U)$. The notation $U\subset \subset D$ means that $\overline{U}$ is compact and contained in $D$. For $U\subset\R^{N}$ and $r>0$ we consider $B_{r}(U):=\{x\in\R^{N}\;:\; \dist(x,U)<r\}$, and we let, as usual  
 $B_r(x)=B_{r}(\{x\})$ be the open ball in $\R^{N}$ centered at $x \in \R^N$ with radius $r>0$. For any subset $M \subset \R^N$, we denote by $1_M: \R^N \to \R$ the
characteristic function of $M$ and by $\diam(M)$ the diameter of $M$. If $M$ is measurable, $|M|$ denotes the Lebesgue measure of $M$. For the unit ball we will use in particular $\omega_N:=|B_1(0)|$. Moreover, if $w: M \to \R$ is a function,  we let $w^+= \max\{w,0\}$ and $w^-=-\min\{w,0\}$ denote the positive and negative part of $w$, resp. 

\subsection{Some elementary inequalities}

We use the notation $a^{\ast q}:=|a|^{q-1}a$ for any $a\in \R$, $q>0$. Note that for $a\geq 0$ we have $a^{\ast q}=a^q$ and for $a<0$ we have $a^{\ast q}=-|a|^q$. Moreover, we have the following elementary inequalities of this function:

\begin{lemma}[see Section 2.2, \cite{IMS14}] For all $b\geq0$, $q>0$
	\begin{align}
	 (a+b)^{q}&\leq \max\{1,2^{q-1}\}\left( a^q+b^q\right)&\text{ if $a\geq0$.}\label{pgeq2klein}\\
	 (a+b)^{\ast q}-a^{\ast q}&\geq 2^{1-q}b^q&\text{ if $a\in \R$, $q\geq 1$.}\label{pgeq2}
	\end{align}
\end{lemma}

\begin{lemma}
	Let $M,q>0$. Then there are $C_{M,1},C_{M,2}>0$ such that for all $a\in[-M,M]$, $b\geq 0$
	\begin{align}
	a^{\ast q}-(a-b)^{\ast q}&\leq C_{M,1} \max\{b,b^{q}\},\quad \label{pgen}\\
	(a+b)^{\ast q}-a^{\ast q}&\geq C_{M,2} \min\{b,b^q\}. \quad\label{pin01}
	\end{align}
\end{lemma}
\begin{proof}

Inequality (\ref{pgen}) is shown in \cite[Section 2.2]{IMS14}. Moreover, if $q\geq1$ then (\ref{pin01}) follows from (\ref{pgeq2}) and indeed no bound on $a$ is needed. For $q\in(0,1)$ fix $M>0$ as stated and let $a\in [-M,M]$. Note that the map $t\mapsto (t+a)^{\ast q}-t^{\ast q}$ satisfies for $t\in[-1,1]$
\[
(t+1)^{\ast q}-t^{\ast q}=q\int_{t}^{t+1}|v|^{q-1}\ dv\geq q\max\{|t|,|t+1|\}^{q-1}\geq q2^{q-1},
\]
since $q<1$. Hence for $b>\max\{0,|a|\}$ we have
\[
\frac{ (a+b)^{\ast q}-a^{\ast q}}{b^{q}}=\left(\frac{a}{b}+1\right)^{\ast q}-\left(\frac{a}{b}\right)^{\ast q}\geq q2^{q-1}.
\]
And for $0\le b\le|a|$ -- using again that $q<1$ -- we have
\[
(a+b)^{\ast q}-a^{\ast q}=qb\int_{0}^{1}|a+vb|^{q-1}\ dv \geq q(|a|+|b|)^{q-1} b\geq q 2^{q-1}M^{q-1} b.
\]	
\end{proof}

\begin{lemma}[see Lemma 2, \cite{L15}]
	For all $q\in(0,1]$ there is $C>0$ such that
	\begin{equation}\label{p01speziell}
	|(a+b)^{\ast q}-a^{\ast q}|\leq C|b|^q \quad\text{ for all $a,b\in \R$.}
	\end{equation}
\end{lemma}

\subsection{Function spaces and their properties}
In the following we let $s\in(0,1)$, $p>1$, and $D\subset \R^N$ open. We denote $\cW^{s,p}_0(D)$ and $\tilde{W}^{s,p}(D)$ as introduced in \eqref{wsp-d}, \eqref{wsp-gen}.
Moreover, we denote formally for function $u,v:\R^N\to\R$
\begin{equation}\label{pairing}
\langle u,v\rangle_{s,p}:= \int_{\R^N}\int_{\R^N} \frac{(u(x)-u(y))^{\ast(p-1)}(v(x)-v(y))}{|x-y|^{N+sp}}\ dxdy.
\end{equation}
Recall that $\cW^{s,p}_0(D)$ is a Banach space with the norm
\[
\|u\|_{W^{s,p}}:=\Big( \|u\|^p_{L^p(D)} + \langle u,u\rangle_{s,p}\Big)^{\frac{1}{p}}
\]
and corresponds to the completion of $C^{\infty}_c(D)$ w.r.t. this norm (see e.g. \cite{G85}).

\begin{lemma}\label{well-defined}
The map $\langle\cdot,\cdot\rangle_{s,p}:\tilde{W}^{s,p}(D) \times \cW^{s,p}_0(D)\to \R$ is well-defined.
\end{lemma}

\begin{proof}
We have by H\"older's inequality with $q=\frac{p}{p-1}$
\begin{align*}
|\langle u,v\rangle_{s,p}|&\leq \int_{\R^N}\int_{\R^N} \frac{|u(x)-u(y)|^{p-1}|v(x)-v(y)|}{|x-y|^{N+sp}}\ dxdy\\
&=\int_{D}\int_{\R^N} \frac{|u(x)-u(y)|^{p-1}|v(x)-v(y)|}{|x-y|^{N+sp}}\ dxdy+\int_{\R^N\setminus D}\int_{D} \frac{|u(x)-u(y)|^{p-1}|v(x)|}{|x-y|^{N+sp}}\ dxdy\\
&\leq \Bigg( \int_{D}\int_{\R^N} \frac{|u(x)-u(y)|^{p}}{|x-y|^{[\frac{N}{q}+s(p-1) ]q}}\ dxdy\Bigg)^{\frac{1}{q}} \Bigg( \int_{D}\int_{\R^N} \frac{|v(x)-v(y)|^{p}}{|x-y|^{[\frac{N}{p}+s]p}}\ dxdy \Bigg)^{\frac{1}{p}}\\
&\qquad  + \Bigg( \int_{\R^N\setminus D}\int_{D} \frac{|u(x)-u(y)|^{p}}{|x-y|^{[\frac{N}{q}+s(p-1) ]q}}\ dxdy\Bigg)^{\frac{1}{q}} \Bigg( \int_{\R^N\setminus D}\int_{D} \frac{|v(x)|^{p}}{|x-y|^{[\frac{N}{p}+s]p}}\ dxdy \Bigg)^{\frac{1}{p}}\\
&\leq 2\Bigg( \int_{D}\int_{\R^N} \frac{|u(x)-u(y)|^{p}}{|x-y|^{N+sp}}\ dxdy\Bigg)^{\frac{1}{q}}\|v\|_{W^{s,p}(\R^N)}<\infty.
\end{align*}
\end{proof}	

\begin{lemma}\label{pos0}
	We have $\tilde{W}^{s,p}(D)$ is a vector space with the following properties:
	\begin{enumerate}
		\item $C^{0,1}_c(\R^N)\subset W^{s,p}(\R^N)\subset \tilde{W}^{s,p}(D)$,
		\item If $u\in \tilde{W}^{s,p}(D)$, then $u^{\pm}\in \tilde{W}^{s,p}(D)$.
	\end{enumerate} 	
\end{lemma}
\begin{proof}
The fact that $\tilde{W}^{s,p}(D)$ is a vector space follows from \eqref{pgeq2klein}. Moreover, we have $C^{0,1}_c(\R^N)\subset W^{s,p}(\R^N)$ (see e.g. \cite{G85}), and $u\in W^{s,p}(\R^N)\subset \tilde{W}^{s,p}(D)$ is trivial. For the second statement, note that we have $|u|\in \tilde{W}^{s,p}(D)$, since
\[
|u(x)-u(y)|\geq \big||u|(x)-|u|(y)\big| \quad \text{for all $x,y\in \R^N$.}
\]
Hence $2u^{\pm}=|u|\pm u\in \tilde{W}^{s,p}(D)$.
\end{proof}

\begin{lemma}\label{restriction}
	Let $D$ be bounded and $u\in \tilde{W}^{s,p}(D)$ with $u=0$ on $\R^N\setminus D$. Then $u\in \cW^{s,p}_0(D)$.
\end{lemma}
\begin{proof}
	Since $D$ is bounded, we have $u\in L^p(D)$. Moreover,
	\begin{align*}
	\int_{\R^N}\int_{\R^N}& \frac{|u(x)-u(y)|^p}{|x-y|^{N+sp}}\ dxdy=\int_{D}\int_{\R^N} \frac{|u(x)-u(y)|^p}{|x-y|^{N+sp}}\ dxdy+ \int_{\R^N\setminus D}\int_{D} \frac{|u(x)|^p}{|x-y|^{N+sp}}\ dxdy\\
	&\leq \int_{D}\int_{\R^N} \frac{|u(x)-u(y)|^p}{|x-y|^{N+sp}}\ dxdy+ \int_{\R^N}\int_{D} \frac{|u(x)-u(y)|^p}{|x-y|^{N+sp}}\ dxdy=2\int_{D}\int_{\R^N} \frac{|u(x)-u(y)|^p}{|x-y|^{N+sp}}\ dxdy,
	\end{align*}
	which is bounded by assumption.
\end{proof}

An immediate consequence of Lemma \ref{restriction} and Lemma \ref{pos0} is
\begin{cor}\label{pos}
	Let $D$ be bounded and $u\in \tilde{W}^{s,p}(D)$ with $u\geq0$ on $\R^N\setminus D$ then $u^{-}\in \cW^{s,p}_0(D)$.
\end{cor}

In the following, we also say that $v\in \tilde{W}^{s,p}(D)$ satisfies (in weak sense)
\[
(-\Delta)^s_p v \geq g \quad\text{ in $D$}
\]	
for $g\in [\cW^{s,p}_0(D)]'$, the dual of $\cW^{s,p}_0(D)$, if for all nonnegative $\varphi\in \cW^{s,p}_{0}(D)$ with compact support in $\R^N$ we have
\[
\langle u,\phi\rangle_{s,p}\geq \int_{\Omega} g(x) \varphi(x) \ dx.
\]
Similarly we use ``$\leq$'' and ``$=$''.

\begin{lemma}[cf. Lemma 2.9, \cite{IMS14}]\label{scaling}
	Let $t>0$, $u\in \tilde{W}^{s,p}(D)$ satisfy $(-\Delta)^{s}_pu=g$ in $D$ for some $g\in [\cW^{s,p}_0(D)]'$. Then the function $v:\R^N\to\R$, $v(x)=u(tx)$ satisfies $v\in \tilde{W}^{s,p}(t^{-1}D)$ and
	\[
	(-\Delta)^{s}_pv =t^{sp}g(t \cdot) \quad \text{ on $t^{-1}D$}
	\]
\end{lemma}
\begin{proof}
Let $\phi\in \cW^{s,p}_0(t^{-1}D)$, then clearly $\phi(\cdot/t)\in  \cW^{s,p}_0(D)$ and
\begin{align*}
\langle v,\phi\rangle_{s,p}&=\int_{\R^N}\int_{\R^N}\frac{(u(x)-u(y))^{\ast(p-1)}(\phi(\frac{x}{t})-\phi(\frac{y}{t}))}{|x/t-y/t|^{N+sp}}\ dxdy\\
&=t^{-N+sp}\int_{\R^N}\int_{\R^N}\frac{(u(x)-u(y))^{\ast(p-1)}(\phi(\frac{x}{t})-\phi(\frac{y}{t}))}{|x-y|^{N+sp}} t^{-2N}\ dxdy\\
&= t^{-N+sp}\int_{D} g(x)\phi(x/t)\ dx=\int_{t^{-1}D} t^{sp}g(tx)\phi(x)\ dx.
\end{align*}	
\end{proof}

\section{Comparison principles}\label{comparison}

The following is a slight variant of the weak maximum principle presented in \cite[Proposition 2.10]{IMS14}, \cite[Lemma 6]{KKP16a}, and \cite[Lemma 9]{LL14}.

\begin{lemma}\label{wmp}
Let $D\subset \R^N$ be an open set, $q\in L^{\infty}(D)$, $q\geq 0$, and $g\in L^{p'}(D)$, where $p'=\frac{p}{p-1}$ . If $v,w\in \tilde{W}^{s,p}(D)$ are super- and subsolution of $(-\Delta)^s_p u+q(x)u^{\ast(p-1)} = g$ resp. such that $v\geq w$ in $\R^N\setminus D$ and 
\[
\liminf_{|x|\to \infty} (v(x)-w(x))\geq0.
\]
Then $v\geq w$ a.e. in $\R^N$.
\end{lemma}

\begin{proof}
 First assume $D$ is bounded and denote $u(x)=v(x)-w(x)$,
	\[
	Q(x,y):=(p-1)\int_0^1 |w(x)-w(y) +t(u(x)-u(y))|^{p-2}\ dt,\quad P(x)=(p-1)\int_0^1 |w(x)+tu(x)|^{p-2}\ dt.
	\]
	Note that $P\geq 0$ on $\R^N$ and $Q(x,y)=Q(y,x)\geq0$ for $(x,y)\in\R^N\times \R^N$. Moreover, if $P(x)=0$, then $w(x)=0=u(x)$ and if $Q(x,y)=0$, then $w(x)=w(y)$ and $u(x)=u(y)$. Note that $u\geq0$ on $\R^N\setminus D$ and hence $u^-\in  \cW^{s,p}_0(D)$ due to Corollary \ref{pos}. Note that for any $\phi\in \cW^{s,p}_0(D)$, $\phi\geq 0$ we have
	\begin{align*}
-\int_D q(x)P(x) u(x)\phi(x)\ dx&= -\int_Dq(x) v^{\ast(p-1)}(x)\phi(x)+ q(x)w^{\ast(p-1)}(x)\phi(x)\ dx\leq \langle v,\phi\rangle_{s,p}-\langle w,\phi\rangle_{s,p}\\
&= \int_{\R^N}\int_{\R^N} \frac{[(v(x)-v(y))^{\ast(p-1)}-(w(x)-w(y))^{\ast(p-1)}](\phi(x)-\phi(y))}{|x-y|^{N+sp}}\ dxdy\\
&=\int_{\R^N}\int_{\R^N} \frac{Q(x,y)}{|x-y|^{N+sp}} (u(x)-u(y))(\phi(x)-\phi(y))\ dxdy.
	\end{align*}
	 Since $D$ is bounded, we have $u^-\in \cW^{s,p}_0(D)$ and hence since
	\[
	[u(x)-u(y)][u^-(x)-u^-(y)]=-2u^+(x)u^-(y) -(u^-(x)-u^-(y))^2\leq 0
	\]
	for $x,y\in \R^N$, we have with $\phi=u^-$ 
	\begin{align*}
	-\|q^-\|_{L^{\infty}(D)}\int_{D}u^-(x)^2 P(x)\ dx&\leq -\int_{\R^N}\int_{\R^N} \frac{Q(x,y)}{|x-y|^{N+sp}} (u^-(x)-u^-(y))^2\ dxdy\\
	&\leq -\int_D u^-(x)^2 \int_{\R^N\setminus D} \frac{2 Q(x,y)}{|x-y|^{N+sp}} \ dy\ dx\leq 0.
	\end{align*}
	Since $q\geq 0$ this implies
	\[
	0=\int_D u^-(x)^2 \int_{\R^N\setminus D}\frac{Q(x,y)}{|x-y|^{N+sp}}  \ dy\ dx,
	\]
	which by an argumentation as in \cite[Lemma 9]{LL14} is only possible if $u^-=0$ a.e.; indeed, we have for a.e. $(x,y)\in D\times (\R^N\setminus D)$ either $u^-(x)=0$ or $Q(x,y)=0$, but in the latter case, we have as mentioned above $u(x)=u(y)$. Since $x\in \supp\ u^-$, $y\in \R^N\setminus D$ we have $u(x)\leq 0\leq u(y)$, so that $u(x)=0=u(y)$. Hence also in the latter case it follows that $u^-(x)=0$. Hence the claim follows for bounded sets.\\
	If $D$ is unbounded, then since 
	\[
	\liminf_{|x|\to \infty} (v(x)-w(x))\geq0.
	\]
	we find for every $\epsilon>0$ a number $R>0$ such that $v_{\epsilon}\geq w$ in $\R^N\setminus D_R$, where $v_{\epsilon}:=v+\epsilon$ and $D_R:=B_R(0)\cap D$. Moreover, for $\phi\in \cW^{s,p}_0(D)$, $\phi\geq0$ with compact support in $\R^N$, using $q\geq0$ and $t\mapsto t^{\ast(p-1)}$ is increasing, we have
\begin{align*}
\langle &v_{\epsilon},\phi\rangle_{s,p}+\int_D q(x)v_{\epsilon}^{\ast(p-1)}(x)\phi(x)\ dx\\
&\geq \int_{\R^N}\int_{\R^N}\frac{(v(x)+\epsilon-(v(y)+\epsilon))^{\ast(p-1)}(\phi(x)-\phi(y))}{|x-y|^{N+sp}}\ dxdy +\int_D q(x)v^{\ast(p-1)}(x)\phi(x)\ dx\\
&=\langle v,\phi\rangle_{s,p} +\int_Dq(x)v^{\ast(p-1)}(x)\phi(x)\ dx\geq \int_D g(x)\phi(x)\ dx,
\end{align*}
	hence $v_{\epsilon}$ is a supersolution of $(-\Delta)^s_p u+q(x)u^{\ast(p-1)} = g$ in $D_R$ while $w$ is a subsolution of this equation in $D_R$. The first part gives $v+\epsilon=v_{\epsilon}\geq w$ in $\R^N$ and since $\epsilon>0$ is arbitrary this finishes the proof.
%
	
\end{proof}

\begin{remark}
	We note that almost with the same proof it is possible to show the following: Let $D\subset \R^N$ be an open bounded set, $q\in L^{\infty}(D)$, $q\geq 0$, and assume $v,w\in \tilde{W}^{s,p}(D)$ satisfy in weak sense
	\[
	\begin{aligned}
	(-\Delta)^s_p v -(-\Delta)^{s}_pw &\geq -q(x)(v-w)^{\ast(p-1)} &&\text{ in $D$}\\
	v&\geq w&&\text{ in $\R^N\setminus D$.}
	\end{aligned}
	\]
	Then $v\geq w$ a.e. in $\R^N$.
\end{remark}

\begin{lemma}[see Lemma 2.8, \cite{IMS14}]\label{nonlocalchar0}
  Let $D\subset \R^N$ be an open bounded set, $K\subset\subset\R^N\setminus D$ and let $h\in L^1_{loc}(\R^N)$. Then for any $v\in \tilde{W}^{s,p}(D)$ we have in weak sense
\[
 (-\Delta)^s_p(v+h1_{K})= (-\Delta)^s_pv+H \quad\text{ in $D$,}
\]
with
\[
 H(x)=2 \int_{K}\frac{((v(x)-v(y))- h(y))^{\ast (p-1)}-(v(x)-v(y))^{\ast (p-1)}}{|x-y|^{N+sp}}\ dy
\]
for a.e. $x\in D$.
\end{lemma}

\begin{lemma}\label{wahlvondelta}
 Let $D \subset \R^N$ be an open bounded set and $K\subset\subset\R^N\setminus D$ with $|K|>0$. 
\begin{enumerate}
 \item If $p\geq 2$ then there is $C>0$ such that for any $v\in \tilde{W}^{s,p}(D)$ and any $\delta\in(0,1]$ we have in weak sense
\begin{equation}\label{nonlocalchar}
\begin{aligned}
 (-\Delta)^s_p(v-\delta 1_K)&\geq (-\Delta)^s_p v + C\delta^{p-1}\;\text{ and } \\
(-\Delta)^s_p(v+\delta 1_K)&\leq (-\Delta)^s_p v - C\delta^{p-1}.
\end{aligned}
\end{equation}
\item If $p\in(1,2)$ then for any $M>0$ there is $C>0$ such that for any $v\in \tilde{W}^{s,p}(D)$ with $\|v\|_{L^{\infty}(\R^N)}\leq M$ and any $\delta\in(0,1]$ we have in weak sense
\begin{equation}\label{nonlocalchar2}
\begin{aligned}
(-\Delta)^s_p(v-\delta 1_K)&\geq (-\Delta)^s_p v + C\delta\;\text{ and } \\
(-\Delta)^s_p(v+\delta 1_K)&\leq (-\Delta)^s_p v - C\delta.
\end{aligned}
\end{equation}
\end{enumerate}
\end{lemma}
\begin{proof}
 By Lemma \ref{nonlocalchar0} we have in weak sense
\[
 (-\Delta)^s_p(v-\delta 1_K)=(-\Delta)^s_pv + 2\int_{K}\frac{(v(x)-v(y) +\delta)^{\ast (p-1)}-(v(x)-v(y))^{\ast (p-1)}}{|x-y|^{N+sp}}\ dy.
\]
We will start by showing the first inequality in \eqref{nonlocalchar} and \eqref{nonlocalchar2}.\\
Case 1: $p\geq2$. Then by inequality (\ref{pgeq2}) we have
\[
 2\int_{K}\frac{(v(x)-v(y) +\delta)^{\ast (p-1)}-(v(x)-v(y))^{\ast (p-1)}}{|x-y|^{N+sp}}\ dy\geq 2^{3-p} \delta^{\ast (p-1)}\int_{K} \frac{1}{|x-y|^{N+sp}}\ dy.
\]
Hence the first inequality in \eqref{nonlocalchar} holds with $C_1= 2^{3-p}|K| \sup_{x\in D,\ y\in K} |x-y|^{-N-sp}$.\\
Case 2: $p\in(1,2)$. Then with \eqref{pin01} we have
\[
 2\int_{K}\frac{(v(x)-v(y) +\delta)^{\ast (p-1)}-(v(x)-v(y))^{\ast (p-1)}}{|x-y|^{N+sp}}\ dy\geq C_{M,2}\delta \int_{K} \frac{1}{|x-y|^{N+sp}}\ dy.
\]
Hence the first inequality in \eqref{nonlocalchar2} holds with $C_2=C_{M,2}|K| \sup_{x\in D,\ y\in K} |x-y|^{-N-sp}$.\\
For the second inequalities in \eqref{nonlocalchar} and \eqref{nonlocalchar2} note that
\begin{align*}
 (-\Delta)^s_p(v-\delta 1_K)&=(-\Delta)^s_pv + 2\int_{K}\frac{(v(x)-v(y) -\delta)^{\ast (p-1)}-(v(x)-v(y))^{\ast (p-1)}}{|x-y|^{N+sp}}\ dy\\
&=(-\Delta)^s_pv - 2\int_{K}\frac{(v(x)-v(y) -\delta+\delta)^{\ast (p-1)}-(v(x)-v(y)-\delta)^{\ast (p-1)}}{|x-y|^{N+sp}}\ dy.
\end{align*}
Hence this part follows similarly.
\end{proof}


\begin{lemma}\label{vglvf}
 Let $D\subset \R^N$ be an open bounded set and let $s\in(0,1)$, $p>1$, and $f\in C^2_c(D)$.
\begin{enumerate}
 \item If $\frac{1}{1-s}<p\leq2$, then for any $v\in \tilde{W}^{s,p}(D)$ there is $C>0$ such that for all $a\in[-1,1]$ in weak sense
\begin{equation}\label{local}
 |(-\Delta)^s_p (v-af)- (-\Delta)^s_p v |\leq  |a|^{p-1} C \quad\text{ in $\supp\ f$.}
\end{equation}
\item If $p\geq 2$, then for any $v\in \tilde{W}^{s,p}(D)\cap  C^{\alpha}_{loc}(D)\cap L^{\infty}(\R^N)$, $\alpha\in(0,1]$ with $\alpha(p-2)>sp-1$ there is $C>0$ such that for all $a\in[-1,1]$ in weak sense.
\begin{equation}\label{local2}
|(-\Delta)^s_p (v-af)- (-\Delta)^s_p v |\leq  |a|C \quad\text{ in $\supp\ f$.}
\end{equation}
\end{enumerate}
\end{lemma}
 \begin{proof}
Let $s\in(0,1)$, $a\in[-1,1]$, fix $f\in C^{2}_c(D)$, $R=2\diam(D)+1$, $U:=\supp\ f$, and $K:=\sum_{|\alpha|\leq 2} \|\partial^{\alpha}f\|_{L^{\infty}(\R^N)}$. Moreover, let $v\in  \tilde{W}^{s,p}(D)$. If $p\in(\frac{1}{1-s},2]$, then we have for any $\phi\in \cW^{s,p}_0(U)$, $\varphi\geq0$, and a constant $C_1>0$ given by \eqref{p01speziell}
\begin{align*}
 &\left|\ \int_{\R^{N}}\int_{\R^{N}}\frac{\left((v(x)-v(y)-a(f(x)-f(y)))^{\ast (p-1)}-(v(x)-v(y))^{\ast (p-1)}\right)(\varphi(x)-\varphi(y))}{|x-y|^{N+sp}}\ dxdy\right|\\
&\leq 2C_1|a|^{p-1}\ \int_{\R^{N}}\varphi(x)\int_{\R^{N}}\frac{|f(x)-f(x+y)|^{p-1}}{|y|^{N+sp}}\ dxdy\\
&\leq 2C_1|a|^{p-1}K^{p-1}\int_{U}\varphi(x)\left(\int_{B_R(0)}|y|^{-N+(1-s)p-1}\ dy+ \int_{\R^{N}\setminus B_{R}(0)}|y|^{-N-sp}\ dy\right)\ dx\\
&\leq 2C_1N\omega_N|a|^{p-1}K^{p-1}\int_{U}\varphi(x)\left(\int_{0}^Rr^{(1-s)p-2}\ dr+ \int_{R}^\infty r^{-1-sp}\ dr\right)\ dx\\
&\leq 2|a|^{p-1}C_1K^{p-1}N\omega_N\left(\frac{1}{(1-s)p-1}+\frac{2}{sp}\right)\int_{U}\varphi(x)\ dx,
\end{align*}
where $\omega_N=|B_1(0)|$. Hence (\ref{local}) holds with $C=2C_1K^{p-1}N\omega_N\left(\frac{1}{(1-s)p-1}+\frac{2}{sp}\right)$.\\

Next let $p\geq 2$, assume additionally that $v\in L^{\infty}(\R^N)\cap C^{\alpha}_{loc}(D)$ for some $\alpha\in(0,1)$ with $\alpha(p-2)>sp-1$ and denote for $x,y\in \R^N$
\[
 Q(x,y):= (p-1)\int_{0}^{1} |v(x)-v(y)-at (f(x)-f(y))|^{p-2}\ dt.
\]
Note that $Q(x,y)=Q(y,x)\geq 0$ for any $x,y\in \R^N$. Moreover, there is $C_2=C_2(p,\|v\|_{L^{\infty}(\R^N)},f)$ and $C_3=C_3(p,\|v\|_{C^{s+\epsilon}(D)},f)$ and such that
\begin{align}
Q(x,y)&\leq C_2 &&\text{ for $x,y\in \R^N$}\label{Q-bound1}\\
Q(x,y)&\leq C_3|x-y|^{\alpha(p-2)} &&\text{ for $x,y\in U$,}\label{Q-bound2}
\end{align}
where we used that $\overline{U}\subset D$.
Moreover, we have $\dist(\supp\ f,D^c)=\delta>0$. Fix $\varphi \in \cW^{s,p}_{0}(U)$, $\varphi\geq0$. Then, since 
\[
 (v(x)-v(y)-a(f(x)-f(y)))^{\ast (p-1)}-(v(x)-v(y))^{\ast (p-1)}=-a Q(x,y)(f(x)-f(y)),
\]
we have
\begin{align*}
\Bigg|\ \int_{\R^{N}}\int_{\R^{N}}&\frac{\left((v(x)-v(y)-a(f(x)-f(y)))^{\ast (p-1)}-(v(x)-v(y))^{\ast (p-1)}\right)(\varphi(x)-\varphi(y))}{|x-y|^{N+sp}}\ dxdy\Bigg|\notag\\
&= |a|\left|\ \int_{\R^{N}}\int_{\R^{N}}\frac{(f(x)-f(y))(\varphi(x)-\varphi(y))}{|x-y|^{N+sp}}Q(x,y)\ dxdy\right|\notag\\
&\leq 2|a| \int_U \phi(x) \Bigg|\ \lim_{\epsilon\to 0}\int_{B^c_{\epsilon}(x)} \frac{(f(x)-f(y)) Q(x,y)}{|x-y|^{N+sp}}dy \Bigg| \ dx.
\end{align*}
We now follow closely the lines of proof \cite[Lemma 3.8]{JW17} to show that
\begin{equation}\label{bound-gen}
\sup_{x\in D}\Bigg|\lim_{\epsilon\to 0}\ \int_{ B^c_{\epsilon}(x)} \frac{(f(x)-f(y)) Q(x,y)}{|x-y|^{N+sp}}dy \Bigg| \leq C_4
\end{equation}
for a constant $C_4=C_4(D,N,p,f,\|v\|_{C^{\alpha}(D)},\|v\|_{L^{\infty}(\R^N)})>0$. Clearly, once \eqref{bound-gen} is shown, this finishes the proof of \eqref{local2}. To see \eqref{bound-gen}, fix $x\in U$ and note that $B_{\delta}(x)\subset D$. Let $\epsilon\in(0,\delta)$. Moreover, since $f\in C^2_c(D)$ we have
\[
|2f(x)-f(x-y)-f(x+y)|\leq K|y|^2 \qquad\text{ for all $x,y\in \R^N$}
\] 
and from \eqref{Q-bound2} and \eqref{pgeq2klein} there is $K_{2}>0$ such that
\[
|Q(x,x-y)-Q(x,x+y)|\leq K_2|y|^{\alpha(p-2)} \qquad\text{ for all $x\in \supp\ f$, $y\in B_{\delta}(0)$.}
\]
Hence with \eqref{Q-bound1} and \eqref{Q-bound2}
\begin{align*}
 \Bigg|\int_{ B^c_{\epsilon}(x)} &\frac{(f(x)-f(y)) Q(x,y)}{|x-y|^{N+sp}}dy\Bigg|=\Bigg| \int_{ B^c_{\epsilon}(0)} \frac{(f(x)-f(x\pm y)) Q(x,x\pm y)}{|y|^{N+sp}}dy\Bigg|\\
&= \Bigg|\frac{1}{2}\int_{ B^c_{\epsilon}(0)} \frac{(2f(x)-f(x+y)-f(x-y)) Q(x,x+y)}{|y|^{N+sp}}dy\\
&\qquad\qquad +\frac{1}{2}\int_{ B^c_{\epsilon}(0)}(f(x)-f(x-y))(Q(x,x+y)-Q(x,x-y))|y|^{-N-sp}\ dy\Bigg|\\
&\leq \frac{C_3K}{2}\int_{ B_{\delta}(0)\setminus B_{\epsilon}(0)} |y|^{2-N-sp+\alpha(p-2)}dy+ C_24K\int_{B_{\delta}^c(0)}|y|^{-N-sp}\ dy\\
&\qquad\qquad +\frac{K_2K}{2}\int_{ B_{\delta}(0)\setminus B_{\epsilon}(0)}|y|^{1+\alpha(p-2)-N-sp}\ dy\\
&\leq \frac{C_3}{2}NK_1\omega_N \int_0^{\delta} \rho^{1-sp+\alpha(p-2)}d\rho+ C_2K4N\omega_N\int_{\delta}^{\infty}\rho^{-1-sp}\ d\rho+\frac{K_2K}{2}N\omega_N\int_{0}^{\delta}\rho^{\alpha(p-2)-sp}\ d\rho\\
&= \frac{C_3KN\omega_N \delta^{2-sp+\alpha(p-2)}}{2(2-sp+\alpha(p-2))} +\frac{C_2K4}{sp}N\omega_N\delta^{-sp}+ \frac{K_2KN\omega_N\delta^{1-2s+\alpha(p-2)}}{2(1-sp+\alpha(p-2))}  =:C_4<\infty,
\end{align*}
where we have used $\alpha(p-2)>sp-1$.
 \end{proof}

\begin{lemma}\label{smpbasis}
 Let $D\subset \R^N$ be an open bounded set, $K\subset\subset \R^N\setminus D$ with $|K|>0$, $\delta\in(0,1]$, $s\in(0,1)$, and $p>1$. Moreover fix $f\in C^{2}_c(D)$ with $0\leq f\leq 1$.
\begin{enumerate}
 \item If $\frac{1}{1-s}<p\leq2$, then for any $v\in \tilde{W}^{s,p}(D)\cap L^{\infty}(\R^N)$ there is $a_0,b>0$ such that in weak sense for all $a\in(0,a_0]$
\begin{equation}\label{lowerbound}
\begin{aligned}
  (-\Delta)^s_p (v-af-\delta 1_{K})&\geq (-\Delta)^s_p v+b &&\text{ in $\supp\ f$ and }\\
  (-\Delta)^s_p v-b&\geq (-\Delta)^s_p (v+af+\delta 1_{K}) &&\text{ in $\supp\ f$.}
  \end{aligned}
\end{equation}
\item If $p\geq 2$, then for any $v\in \tilde{W}^{s,p}(D)\cap  C^{\alpha}(D)\cap L^{\infty}(\R^N)$, $\alpha\in(0,1]$ with $\alpha(p-2)>sp-1$ there is $a_0,b>0$ such that (\ref{lowerbound}) holds in weak sense for all $a\in(0,a_0]$.
\end{enumerate}
\end{lemma}
\begin{proof}
 Fix $v,s,p$ as stated. By Lemma \ref{wahlvondelta} we have in all cases
\[
 (-\Delta)^s_p (v-af-\delta 1_{K})\geq (-\Delta)^s_p (v-af) + C\min\{\delta,\delta^{p-1}\}
\]
for some $C>0$. Moreover, by Lemma \ref{vglvf} we have for some $\tilde{C}>0$ in weak sense
\[
  (-\Delta)^s_p (v-af) \geq (-\Delta)^s_p v- \max\{a,a^{p-1}\}\tilde{C},\quad a\in(0,1].
\]
Hence we may fix $a_0=a_0(\delta)\in(0,1]$ such that $b=- \max\{a_0,a_0^{p-1}\}\tilde{C}+C\min\{\delta,\delta^{p-1}\}>0$. This shows the first inequality in (\ref{lowerbound}). The second inequality in (\ref{lowerbound}) follows similarly.
\end{proof}

\begin{proof}[Proof of Theorem \ref{smp}]
	We start with the case $v\in L^{\infty}(\R^N)$. Assume there is $M\subset\subset\R^{N}$ with $0<|M|<|D|$ and $\delta:=\underset{M}\essinf\ (v-w)>0$. Without loss we may assume $\delta\leq 1$. Fix $K\subset\subset D\setminus M$ and let $f\in C^{2}_c(D\setminus M)$ be given with $f\equiv 1$ in $K$ and $0\leq f\leq 1$. Let $a_0,b>0$ be given by Lemma \ref{smpbasis} and fix $a\in(0,a_0]$. Let $u_a:=v-af-\delta 1_{M}$ and note that $u\in \tilde{W}^{s,p}(\supp\ f)$. Then -- assuming for 2. in addition $v\in C^{\alpha}_{loc}(D)$ for $\alpha\in(0,1]$ as stated -- we have in weak sense in $\supp\ f$ by Lemma \ref{smpbasis}
	\begin{align*}
	(-\Delta)^s_pu_a&\geq (-\Delta)^s_p v+b \geq -q(x)v^{\ast(p-1)} +b\\
	&\geq  -q(x)u_a^{\ast(p-1)} +b +q(x)\Big((v-af)^{\ast(p-1)} -v^{\ast(p-1)} \Big)\\
	&\geq -q(x)u_a^{\ast(p-1)} +b -\|q\|_{L^{\infty}(D)}\Big( v^{\ast(p-1)}-(v-a)^{\ast(p-1)} \Big).
	\end{align*}
	By \eqref{p01speziell} if $p\leq2$ or by \eqref{pgen} with $M=\|v\|_{L^{\infty}(\R^N)}$, there is $C>0$ depending only on $p$ and if $p>2$ on $M$, such that
	\[
	\Big( v^{\ast(p-1)}-(v-a)^{\ast(p-1)} \Big)\leq C \max\{a,a^{p-1}\}.
	\]
	It follows that
	\[
	b -\|q\|_{L^{\infty}(D)}\Big( v^{\ast(p-1)}-(v-a)^{\ast(p-1)} \Big)\geq b -\|q\|_{L^{\infty}(D)}C\max\{a,a^{p-1}\}.
	\]
	Hence we may choose $a_1\in(0,a_0]$ such that $ b -\|q\|_{L^{\infty}(D)}C\max\{a_1,a_1^{p-1}\}>0$. Then $u_{a_1}$ satisfies
	\[
	(-\Delta)^s_pu_{a_1} +q(x)u_{a_1}^{\ast(p-1)}\geq g \quad\text{ in $\supp\ f$}
	\] 
	and $u_{a_1}\geq w$ in $\R^N\setminus \supp \ f$. Lemma \ref{wmp} implies $u_{a_1}\geq w$ a.e. in $\supp f$, and hence $v\geq w+_{a_1}f$ in $\supp\ f$. In particular,
	\[
	v\geq w+_{a_1} \quad\text{ in $K$.}
	\]
	Since $K,f$ were chosen arbitrarily, this shows 
	\[
	\underset{K}\essinf\ (v-w)>0 \quad\text{ for all $K\subset\subset D\setminus M$.}
	\]
	Since $0<|M|<|D|$ we may fix $\tilde{M}\subset \subset D\setminus M$ with $\tilde{\delta}:=\underset{\tilde{M}}\essinf\ (v-w)>0$. Repeating the above argument now with $D\setminus \tilde{M}$ in place of $D\setminus M$  this shows the claim in case 1 and also in case 2 with $v\in  C_{loc}^{\alpha}(D)\cap L^{\infty}(\R^N)$.\\
	If $w\in C_{loc}^{\alpha}(D)\cap L^{\infty}(\R^N)$ we note that Lemma \ref{smpbasis} implies also the existence of $a_0,b>0$ such that
	\[
	(-\Delta)^s_p (w+af+\delta 1_{K})\leq (-\Delta)^s_p w-b.
	\]
	for all $a\in(0,a_0]$. Proceeding as in the first case we find $v\geq w+af$ in $\supp f$ and since $f$ was chosen arbitrarily, this finishes the proof similarly as in the first case.
\end{proof}

\begin{remark}
	We note that as for the weak maximum principle, it is also with almost the same proof possible to show: Let $D\subset \R^N$ be an open set, $q\in L^{\infty}(D)$, $q\geq 0$, $s\in(0,1)$, $p>1$, and assume $v,w\in \tilde{W}^{s,p}(D)$ satisfy in weak sense
	\[
	\begin{aligned}
	(-\Delta)^s_p v -(-\Delta)^{s}_pw &\geq -q(x)(v-w)^{\ast(p-1)} &&\text{ in $D$}\\
	v&\geq w&&\text{ in $\R^N$.}
	\end{aligned}
	\]
	If one of the following holds
	\begin{enumerate} 
		\item $\frac{1}{1-s}<p\leq2$ and $v\in L^{\infty}(\R^N)$ or $w\in L^{\infty}(\R^N)$, or
		\item  $p\geq 2$ and for some $\alpha\in(0,1]$ with $\alpha(p-2)>sp-1$ we have $v\in C_{loc}^{\alpha}(D)\cap L^{\infty}(\R^N)$ or $w\in C^{s}_{loc}(D)\cap L^{\infty}(\R^N)$,
	\end{enumerate}
	then either $v\equiv w$ a.e. in $\R^N$ or  $\underset{K}\essinf\ (v-w)>0$ for all $K\subset\subset D$.
\end{remark}

\begin{cor}[Strong maximum principle]
 Let $D\subset \R^N$ be an open set and let $f:D\times \R\to [0,\infty)$ be such that $f(x,0)=0$ for all $x\in D$ and $$\int_{K}f(x,v(x))\varphi(x)\ dx<\infty\quad\text{ for all $K\subset\subset D$ and $v,\varphi\in L^p(K)$, $v,\varphi\geq0$.}$$
 Moreover, let $v\in\tilde{W}^{s,p}(D)$ be a nonnegative function satisfying in weak sense
\[
 (-\Delta)^s_p v\geq f(x,v) \quad\text{ in $D$.}
\]
If $p>\frac{1}{1-s}$, then either $v\equiv 0$ in $\R^N$ or $\underset{K}\essinf\ v>0$ for all $K\subset\subset \Omega$.
\end{cor}
\begin{proof}
 Since $0\in W^{s,p}(\R^N)\cap C_c^{0,1}(\R^N)$ satisfies $(-\Delta)^s_p 0=0\leq f(x,v)\leq (-\Delta)^s_p v$ an application of Theorem \ref{smp} proves the claim.
\end{proof}

\section{Starshaped superlevel sets}\label{applications}

As in \cite{JKS17} we use the following observation.

\begin{lemma}[See Lemma 3.1, \cite{JKS17}]\label{lmstar}
	Let  $u:\R^N\to\R$ such that $M=\max_{\R^N}u=u(0)$. Then the superlevel sets $U(\ell)$, $\ell\in\R$, of $u$ are all (strictly) starshaped if and only if $u(tx)\leq u(x)$ ($u(tx)<u(x)$) for every $x\in\rn$ and every $t> 1$.
\end{lemma}

\begin{thm}\label{res2}
	Let $D=D_0\setminus \overline{D}_1$ with $D_0, D_1\subset \R^N$ open bounded sets such that $0\in D_1$ and $\overline{D}_1\subset D_0$. Moreover, let $b_0,b_1\in L^{\infty}(\R^N)\cap \tilde{W}^{s,p}(\R^N)$ such that $b_0\equiv 0$ on $\partial D_0$ and $b_1\equiv 1$ on $\partial D_1$ and let $q\in L^{\infty}(D)$ be a nonnegative function, $g\in L^{p'}(D)$ with $p'=\frac{p}{p-1}$ such that both functions satisfy (A2), i.e.
	\begin{enumerate}
		\item For all $t\geq 1$ and $x\in \R^N$ such that $tx\in D_0\setminus \overline{D}_1$ we have $t^{sp}q(tx)\geq q(x)$ and $t^{sp}g(tx)\geq g(x)$.
	\end{enumerate}
	 Let $u\in \tilde{W}^{s,p}(\R^N)$ be a continuous solution of
	\begin{equation}\label{main-prob4}
	\left\{\begin{aligned}
	(-\Delta)^s_pu+q(x)|u|^{p-2}u&=-g&&\text{ in $D$;}\\
	u&=b_0&&\text{ on $D_0^c$;}\\
	u&=b_1&& \text{ on $D_1$,}
	\end{aligned}
	\right.
	\end{equation}
	such that $0\leq u\leq 1$ in $D$. If $b_0|_{D_0^c}$, $b_1|_{D_1}$ have starshaped superlevel sets, then the following holds:
	\begin{enumerate}
			\item If $D_0$ and $D_1$ are starshaped sets, then the superlevel sets $U(\ell)$ of $u$ are  starshaped for $\ell\in (0,1)$.
			\item If $D_0$ and $D_1$ are strictly starshaped sets, $0<u<1$ in $D$ and 
		\begin{enumerate}
			\item $\frac{1}{1-s}<p\leq 2$ \underline{or}
			\item $p\geq 2$ and $u\in C^{\alpha}_{loc}(D)$ for some $\alpha\in(0,1]$ with $\alpha(p-2)>sp-1$,
		\end{enumerate}		
			then the superlevel sets $U(\ell)$ of $u$ are strictly starshaped for $\ell\in (0,1)$.
		\end{enumerate}	
\end{thm}

\begin{proof}
	We proceed as in the proof of \cite[Theorem 1.8]{JKS17}. Let $D_0$, $D_1$, $b_0$, $b_1$, and $u$ be given as for 1. Note that by assumption it follows that $u\in L^{\infty}(\R^N)$. Denote for any $t>1$ and any function $v:\R^N\to \R$ 
	$$v_t(x)=v(tx)\quad x\in\R^N\,.$$
	Thanks to Lemma \ref{lmstar}, the starshapedness of the level sets of $u$ is equivalent to  
	\begin{equation}\label{ut<0}
	u\geq u_t\,\,\text{ in }\R^N\,\,\text{ for }\,\,t>1\,.
	\end{equation}
	Observe that since the superlevel sets of $b_0$ and $b_1$ are starshaped and $0\leq u\leq 1$ in $D$, we have $u\geq u_t$ in $\R^N\setminus D_0$ and in $t^{-1}\overline{D}_1$ and
	\begin{equation}\label{utboundary}
	u(x)\geq  u_t(x)\quad\text{for }x\in D_0\setminus (t^{-1} D_0)\,\text{ and }\, x\in \overline{D}_1\setminus(t^{-1}\overline{D}_1)\,.
	\end{equation}
	
	Put $D_t=(t^{-1}D_0)\setminus\overline{D}_1$. It remains to investigate $u_t$ in $D_t$. Note that since $D_0$ is bounded, $D_t$ is empty for $t$ large enough. By Lemma \ref{scaling} we get (in weak sense) 
	\[
	(-\Delta)^{s}_pu_t =t^{sp} \big[(-\Delta)^{s}_p u\big]_t=-t^{sp}q(tx)u_t^{\ast(p-1)}-t^{sp}g(tx)\leq -q(x)u_t^{\ast(p-1)}+g(x)  \qquad\text{ in $D_t$,}
	\]
	where we used that $u_t\geq0$ in $\R^N$. Hence Lemma \ref{wmp} implies $u\geq u_t$ in $\R^N$. This proves 1.\\
	If in addition $u$, $D_0$ and $D_1$ satisfy the assumptions of 2., then observe that with the same argument as above Theorem \ref{smp} yields either $u_t\equiv u$ in $\R^N$ or $u>u_t$ in $D_t$. Since $u\equiv u_t$ in $\R^N$ is not possible for $t>1$ due to the strict inequality $0<u<1$ in $D$, we must have $u>u_t$ in $D_t$ for all $t>1$. This proves 2.
\end{proof}

\begin{proof}[Proof of Theorem \ref{res1}]
If $D_0$ and $D_1$ are starshaped, then the result follows immediately from Theorem \ref{res2}.1. Hence let $D_0$ and $D_1$ be strictly starshaped. Since the functions $v\equiv 1$ satisfies $(-\Delta)^{s}_pv+q(x)|v|^{p-2}v=q(x)\geq 0$ in $D$ and $v\geq u$, $v\not \equiv u$ in $D^c$, Theorem \ref{smp} implies $u<1$ in $D$. Similarly, the function $w\equiv 0$ satisfies $(-\Delta)^{s}_pw+q(x)|w|^{p-2}w=0$ in $D$ and $u\geq w$, $u\not\equiv w$ in $D^c$, Theorem \ref{smp} implies $u>0$ in $D$. Hence the claim follows from Theorem \ref{res2}.2 with $b_0\equiv0$ and $b_1\equiv 1$.
\end{proof}

\begin{proof}[Proof of Theorem \ref{res3}]
Let $q,u$ be as stated and assume $u\not\equiv 0$ on $\R^N$. Then by the strong comparison principle we must have $u>0$ on $\R^N_+$ since $0$ is solution of $(-\Delta)^{s}_p u+q(x)u^{\ast(p-1)}=0$ and $u\geq0$. For $t\geq0$ denote $u_t(x):=u(x+te_1)$ for $x\in \R^N$ and $H_t:=\{x\in \R^N_+\;:\; x_1>t\}$. Then we have in weak sense for all $t\geq0$
\[
(-\Delta)^{s}_pu+q(x)u^{\ast(p-1)} \geq (-\Delta)^s_pu_t(x)+q(x)u_t^{\ast(p-1)} \qquad \text{ on $H_t$}
\]
by the assumptions on $q$. Hence for all $t\geq0$, since $\lim\limits_{|x|\to \infty}(u(x)-u_t(x))=0$, Lemma \ref{wmp} implies $u_t\leq u$ on $\Omega_t$. We claim
\begin{equation}\label{claim1}
\text{ For all $t>0$ we have $u_t<u$ on $H_t$}
\end{equation}
Fix $t>0$ and assume by contradiction that $u_t\equiv u$ on $\R^N$. But then $u\equiv 0$ on $\{x\in \R^N\;:\; 0<x_1<t\}$, which is a contradiction to the fact that $u>0$ on $\R^N_+$. Hence $u_t\not \equiv u$ on $\R^N$ and Theorem \ref{smp} implies \eqref{claim1} since $t>0$ is arbitrary.\\
Note that \eqref{claim1} implies that $u$ is strictly decreasing in $x_1$, but since $u\geq 0$, $u=0$ on $(\R^N_+)^c$, and $u$ is continuous, this is a contradiction and hence we must have $u\equiv 0$ on $\R^N$ as claimed.
\end{proof}

\bibliographystyle{amsplain}

\end{document}